\documentclass{amsart}
\usepackage[english]{babel}

\usepackage{amssymb} 
\usepackage{mathtools}
\usepackage{amsthm}
\usepackage{amsmath}
\usepackage{thmtools}
\usepackage{xcolor}
\usepackage{cancel}

\usepackage{hyphenat} 
\usepackage{csquotes}
\usepackage{marginnote} 
\usepackage{hyperref}
\hypersetup{
    final=true,
    plainpages=false,
    pdfstartview=FitV,
    pdftoolbar=true,
    pdfmenubar=true,
    bookmarksopen=true,
    bookmarksnumbered=true,
    breaklinks=true,
    linktocpage,
    colorlinks=true,
    linkcolor=teal,
    urlcolor=teal,
    citecolor=teal,
    anchorcolor=green
  }

\usepackage[
style = alphabetic, 
sorting=nyt, 
url=true,
maxnames=10,
minnames=10, 
datamodel=standard,
maxalphanames=10]{biblatex}
\DeclareFieldFormat[article]{volume}{\mkbibbold{#1}\addcolon \addspace}
\DeclareFieldFormat[article]{number}{\mkbibparens{#1}}
\renewbibmacro*{in:}{}

\numberwithin{equation}{section}

\theoremstyle{definition}
\newtheorem{definition}{}[section]

\newtheorem{proposition}[definition]{}

\newtheorem{theorem}[definition]{}

\newtheorem{corollary}[definition]{}

\newtheorem{lemma}[definition]{}

\newtheorem{remark}[definition]{}

\newtheorem{example}[definition]{}

\newtheorem{conjecture}[definition]{}

\renewcommand{\eqref}[1]{\hyperref[#1]{(\protect\NoHyper\ref{#1}\protect\endNoHyper)}}

\DeclareCiteCommand{\cite}
{}
{\bibhyperref{[{\protect\NoHyper\usebibmacro{cite}}\protect\endNoHyper\usebibmacro{postnote}]}}
{   
    \multicitedelim%
}
{}

\title{
On Weyl structures reducible in the direction of the Lee form
}
\date{}

\author[J.~L.~Carmona Jim\'enez]{José Luis  {\small CARMONA JIMÉNEZ}}
\address{JLCJ: 
Institute of Mathematics “Simion Stoilow” of the Romanian Academy, 21 Calea Grivitei, 010702 Bucharest, Romania}
\email{jcarmona@imar.ro}

\thanks{
\noindent JLCJ has been supported by the PNRR-III-C9-2023-I8 grant CF 149/31.07.2023 {\em Conformal Aspects of Geometry and Dynamics} and by the
projects PID2021-126124NB-I00 (AEI, Spain) and PID2024-156578NB-I00 (Spain).}

\begin{document}

\begin{abstract}
    A Weyl structure on a Riemannian manifold $(M,g)$ is a torsion-free linear connection $\nabla$ such that there is a $1$-form $\theta$ (called the Lee form) satisfying $\nabla g = 2\, \theta \otimes g$.
    We examine the case in which there exists a $\nabla$-parallel distribution of codimension $1$ on which the Lee form vanishes identically.
    We prove that if $(M,g)$ is complete with $\theta$ closed, then the Weyl structure must be flat or exact.
    We apply this to prove the conjecture in~\cite{L2023}, namely, every homogeneous Kenmotsu manifold is isometric to the real hyperbolic space.
\end{abstract}

\maketitle
{\footnotesize \noindent
\textbf{Key words.} Locally conformally product, Weyl structure, holonomy, Kenmotsu manifolds.\\
\textbf{MSC2020:} 53C05, 53C18, 53C30, 53C15, 53D15.} 

\section{Introduction}
A fundamental question in conformal geometry concerns the conditions under which a Riemannian manifold can be viewed, from a conformal perspective, as a local product of manifolds.
This question was formally introduced in~\cite{BM2011}, where the concept of a conformal product was defined in terms of Weyl structures.
A Weyl structure on a Riemannian manifold $(M,g)$ is a torsion-free linear connection $\nabla$ for which there exists a $1$-form $\theta$ (called the Lee form) satisfying $\nabla g = 2\, \theta \otimes g$.
A Weyl structure $\nabla$ is called closed (resp.~exact) if $\theta$ is closed (resp.~exact), and is reducible if there exists a non-trivial $\nabla$-parallel distribution of the tangent bundle.

In recent years, considerable progress has been made in this direction, especially for compact manifolds.
For instance,~\cite{BM2016} conjectured that every closed Weyl structure on a compact manifold is flat, exact, or has irreducible holonomy.
However,~\cite{MN2015} provided a counter-example, showing that the conjecture needed refinement.
Therefore, an additional case must be included in the classification.
Specifically, the universal cover might be isometric to the product of two manifolds, one flat and the other geodesically incomplete with respect to the lifted connection $\tilde{\nabla}$.
Finally, the refined conjecture was proved in~\cite[Thm.~1.5]{K2019}.

This last case motivated the investigations conducted in~\cite{F2024}, where the author defines a \emph{locally conformally product structure} (LCP) as a reducible, non-flat, and incomplete metric on the universal cover of a compact manifold.
In our setting, it is equivalently defined as a reducible, non-flat, incomplete, and closed Weyl structure $\nabla$ on the manifold.
This research program began only recently~\cite{F2024}.
Since then, it has rapidly developed into a highly active and fruitful area of research.
Recent works include locally conformal product structures on compact Kähler manifolds~\cite{MP2025} and on compact Einstein manifolds~\cite{MP2025*}.
Additionally, Weyl structures with special holonomy on compact conformal manifolds have been studied in~\cite{BFM2025}.

In this article, our goal is to extend the results from the compact case to the complete setting, obtaining analogous conclusions under the assumption that the holonomy of $\nabla$ is reducible.
To this end, we adopt the following framework.
Let $(M,g)$ be a complete Riemannian manifold with a closed Weyl structure $\nabla$.
We study the instances where there exists a codimension $1$ $\nabla$-parallel distribution on which the Lee form vanishes identically.
We shall refer to this situation by saying that \emph{$\nabla$ is reducible in the direction of the Lee form}.
In Section~\ref{sec:def}, we motivate the framework through explicit examples such as manifolds carrying generalized imaginary Killing spinors~(\ref{ex:3.2}) and Kenmotsu manifolds~(\ref{ex:3.3}).

Our main result is as follows.

\begin{theorem}\label{thm:1.1}
    Let $(M,g)$ be a complete Riemannian manifold, and let $\nabla$ be a closed Weyl structure reducible in the direction of the Lee form.
    Then $\nabla$ is flat or exact.
\end{theorem}

One application of~\ref{thm:1.1} concerns Kenmotsu manifolds, a distinguished subclass of almost contact metric manifolds (see~\ref{ex:3.3} and~\ref{def:3.4}).
\cite{L2023} observed that, although the local structure of Kenmotsu manifolds is well understood (see \cite[Thm.~4]{K1972}), the existing literature does not offer significant insights into their global topology beyond the fact that they are non-compact.
We establish constraints in~\ref{cor:5.1} on the global structure of any complete Kenmotsu manifold.
This result implies that complete Kenmotsu manifolds must be globally conformal to a coKähler manifold or locally conformally flat.
In the homogeneous setting, \cite{L2023} proved that every homogeneous Kenmotsu manifold is contractible and proposed the following conjecture:

\begin{conjecture}
Every homogeneous Kenmotsu manifold is isometric to the real hyperbolic space.
\end{conjecture}

Motivated by this conjecture, we study Riemannian homogeneous manifolds whose codimension $1$ distribution (on which the Lee form vanishes) is invariant.
As a consequence,~\ref{thm:5.5} shows that these homogeneous manifolds are homothetic to the real hyperbolic space.
In particular, every homogeneous Kenmotsu manifold is necessarily isometric to the real hyperbolic space; hence, the conjecture in~\cite{L2023} holds.

\subsection*{Structure of the article}
Section~\ref{sec:prelim} provides the necessary background on conformal geometry, Weyl structures, Lee forms, and notation that will be used throughout the text.
Section~\ref{sec:def} introduces the key notion of a Weyl structure reducible in the direction of the Lee form~(\ref{def:3.1}) and illustrates it with two guiding families of examples: manifolds carrying generalized imaginary Killing spinors and Kenmotsu manifolds.
Section~\ref{sec:main} develops the heart of the argument.
After establishing a warped-product splitting of the universal cover~(\ref{lem:4.1}), it proves~\ref{thm:1.1}: any closed Weyl structure on a complete manifold satisfying the reducibility condition must be flat or exact.
Finally, Section~\ref{sec:apps} classifies complete Kenmotsu manifolds as globally conformal to a coKähler manifold or locally conformally flat~(\ref{cor:5.1}), and homogeneous Kenmotsu manifolds as isometric to the real hyperbolic space~(\ref{cor:5.7}).

\subsection*{Acknowledgments} I would like to thank Andrei Moroianu for his valuable comments and observations, and Sergiu Moroianu for fruitful discussions. I would also like to thank the anonymous referee for his/her
attentive reading of the manuscript and helpful remarks.


\section{Preliminaries}\label{sec:prelim}

Let $(M,g)$ be a Riemannian manifold of dimension $n \ge 2$.
We denote its Levi–Civita connection by $\nabla^g$, and we denote the musical isomorphisms determined by the metric $g$ by
\[
(\cdot)^\sharp \colon T^*M \longrightarrow TM,
\qquad
(\cdot)^\flat \colon TM \longrightarrow T^*M.
\]
Recall that these maps are parallel with respect to $\nabla^g$ and satisfy $
(\cdot)^\sharp\circ(\cdot)^\flat = \mathrm{Id}_{TM}$ and $(\cdot)^\flat\circ(\cdot)^\sharp = \mathrm{Id}_{T^*M}$.

A \emph{conformal structure} on a manifold $M$ is an equivalence class of metrics $[g]$ where two metrics $g$ and $g'$ are equivalent if there exists a smooth function $f \in \mathcal{C}^{\infty}(M)$ such that $g' = e^{2f} g$.
A \emph{Weyl structure} on $(M, [g])$ is a torsion-free linear connection $\nabla$ such that for all $h \in [g]$, there exists a $1$-form $\theta_h$ on $M$ satisfying $\nabla h = 2\, \theta_h \otimes h$.
The $1$-form $\theta_h$ is called the \emph{Lee form} of $\nabla$ with respect to $h$.
Therefore, the Weyl structure preserves the conformal class.
If $f \in \mathcal{C}^{\infty}(M)$, then $\theta_{e^{2f}h} = \theta_h + df$.
Furthermore, if we fix a metric $h \in [g]$, the Weyl structure $\nabla$ is uniquely determined by $\theta_h$ via
\begin{align}\label{eq:2.1}
    \nabla _X Y = \nabla^{h} _X Y - \theta_h(Y) \,X + h(X,Y)\,\theta_h ^{\sharp} - \theta_h (X) \,Y,
\end{align}
for all $X$, $Y\in \mathfrak{X}(M)$,
where $\nabla^h$ is the Levi-Civita connection on $(M,h)$.
A Weyl structure is called \emph{closed} (resp.~\emph{exact}) if its Lee form is closed (resp.~exact).
A Weyl structure $\nabla$ is said to be \emph{flat} if its curvature vanishes.

A Weyl structure $\nabla$ on $M$ is said to be \emph{reducible} (or to have \emph{reducible holonomy}) if there is a non-trivial $\nabla$-parallel distribution $\mathcal{D}$.
Let $\mathcal{I} = \mathcal{D}^\perp$ be the orthogonal distribution of $\mathcal{D}$.
Notice that, because $\nabla g = 2\,\theta \otimes g$, we have
\[
2\theta(X)\, g(\mathcal{I}, \mathcal{D}) = X (g(\mathcal{I}, \mathcal{D})) - g(\nabla_X \mathcal{I},\mathcal{D}) - g(\mathcal{I},\nabla_X\mathcal{D})
\]
for all $X\in \mathfrak{X}(M)$.
Thus, using $g(\mathcal{I}, \mathcal{D})=0$ and $\nabla _X \mathcal{D} \subset \mathcal{D}$, we deduce that $\mathcal{I}$ is $\nabla$-parallel.
Furthermore, for any torsion-free connection, every parallel distribution is involutive.
Thus, $\mathcal{D}$ and $\mathcal{I}$ are integrable.

A \emph{locally conformally product (LCP) structure} on a compact manifold $M$ is a conformal structure $c$ paired with a closed, non-exact, and non-flat Weyl structure $\nabla$ with reducible holonomy.
Such structures are equivalently characterized by a reducible Riemannian metric $h$ on the universal cover $\tilde{M}$ of $M$, which depends on $\nabla$.
They are uniquely defined up to a constant factor, so that the fundamental group $\pi_1(M)$ acts by similarities; see \cite[Prop.~2.4]{F2024}.

We briefly present this construction without the compactness assumption.
Let $(M,g)$ be a Riemannian manifold equipped with a closed Weyl structure $\nabla$ such that $\nabla g = 2\, \theta \otimes g$, that is, $\theta$ is the Lee form of $\nabla$ with respect to $g$.
Let $(\tilde{M}, \tilde{g})$ be the universal cover of $(M,g)$.
We denote the lifts of $\theta$ and $\nabla$ by $\tilde{\theta}$ and $\tilde{\nabla}$, respectively.
Since $\theta$ is closed and $\tilde{M}$ is simply connected, it follows that $\tilde\theta$ is exact, i.e.~there exists $f\in \mathcal{C}^{\infty}(\tilde{M})$ satisfying $\tilde{\theta} = df$.
Then $(\tilde{M}, h \coloneqq e^{-2f} \tilde{g})$ is a Riemannian manifold.
Since $\nabla g = 2\, \theta \otimes g$, we obtain  $\tilde \nabla (e^{-2f}\tilde g) = e^{-2f} ( -2 \,\tilde \theta \otimes \tilde g + 2\,\tilde \theta \otimes \tilde g) = 0$, i.e., $\tilde \nabla$ is the Levi-Civita connection of $(\tilde{M}, h \coloneqq e^{-2f} \tilde{g})$.
We say that $\phi \colon \tilde{M} \longrightarrow \tilde{M}$ is a similarity of $(\tilde{M},h)$ of ratio $\lambda \in (0,\infty)$ if it is a diffeomorphism satisfying $\phi^* h = \lambda^2 h$.
Thus, we denote the group of similarities by 
\[
\mathrm{Sim}(\tilde{M}, h) = \{ \phi \in \mathrm{Diff}(\tilde{M}) :\: \exists \lambda_\phi \in (0,\infty)  ,\, \phi^* h = \lambda_\phi ^2\, h \}.
\] 
Finally, in this case, $\mathrm{Sim}(\tilde{M}, h)$ coincides with the Lie group $\mathrm{Aff}(\tilde{M}, \tilde{\nabla})$ of affine transformations of $\tilde{\nabla}$.
Let $\gamma \in \pi _1 (M)$ be a closed curve in $M$.
The deck transformation of $\gamma$ is a diffeomorphism $T_{\gamma}\colon \tilde{M} \longrightarrow \tilde{M}$ that satisfies
\begin{equation*}
    (T_{\gamma})^* \tilde{g} = \tilde{g}, \qquad (T_{\gamma})_* \left(\tilde{\nabla} _X Y\right) = \tilde{\nabla} _{(T_{\gamma})_*X} (T_{\gamma})_*Y
\end{equation*}
for all $X$, $Y\in\mathfrak{X}(\tilde{M})$.
Thus, every deck transformation is a similarity of $(\tilde{M}, h)$ and becomes an isometry precisely when the Lee form $\theta$ is exact.

\section{Definition and examples}\label{sec:def}

\begin{definition}\label{def:3.1}
    Let $(M,g)$ be a Riemannian manifold and $\nabla$ a Weyl structure with non-zero Lee form $\theta$.
    We say \emph{$\nabla$ is reducible in the direction of the Lee form} if one of the following two occurs:
    \begin{enumerate}
        \item there exists a $\nabla$-parallel distribution $\mathcal{D}$ of codimension $1$ such that $\theta (X)= 0$, for all $X\in \mathcal{D}$.
        \item there exists a $\nabla$-parallel distribution $\mathcal{I}$ of dimension $1$ such that $\theta^\sharp\in \mathcal{I}$.
    \end{enumerate}
\end{definition}

Any distribution that is $\nabla$-parallel has a $\nabla$-parallel orthogonal complement (see the third paragraph of Section~\ref{sec:prelim}).
Thus, conditions (1) and (2) in~\ref{def:3.1} are equivalent.
To motivate~\ref{def:3.1} we present two examples.

\begin{example}\label{ex:3.2}
    Let $(M,g)$ be a complete connected Riemannian spin manifold equipped with a spinor $\psi \in \Gamma (\Sigma M)$.
    We say that $\psi$ is~\emph{a generalized imaginary Killing spinor} if there exists a non-zero function $\mu \in \mathcal{C}^\infty (M)$ such that
    \[
    \nabla ^g _X \psi = i \mu \, X \cdot \psi
    \]
    for all $X \in \mathfrak{X}(M)$, where $\nabla^g$ denotes the spinor derivative with respect to the Levi-Civita connection and Clifford multiplication is represented by `` $\cdot$ ".
    Let $\eta (X) := i \langle \psi , X\cdot \psi \rangle$, $\xi  := \eta ^\sharp$, and let $\mathrm{dist}$ denote the distance on the spinor bundle induced by the metric on spinors.
    T.~Friedrich observed in~\cite[Prop.~5]{F1990} that the function $q_\psi = \langle \psi , \psi \rangle^2 - \| \xi \|^2 $ is constant and non-negative.
    Moreover, he showed that $q_\psi$ equals $\langle \psi , \psi \rangle \,\mathrm{dist} (i\psi, \mathcal{V}_\psi)$, where $\mathcal{V}_\psi = \{X\cdot \psi : X \in \mathfrak{X}(M)\}$.
    Thus, we say that $\psi$ is of \emph{type~$1$} if $q_\psi=0$, or, equivalently, there exists a unit vector field $\xi \in \mathfrak{X}(M)$ such that $\xi \cdot \psi = i \psi$.
    Otherwise, if $q_\psi >0$, $\psi$ is of \emph{type~$2$}.

    Generalized imaginary Killing spinors were first studied in~\cite{B1989} and~\cite{B1989*} for constant $\mu$, and later, for non-constant $\mu$ in~\cite{R1991}.
    They have interesting geometric properties.
    For example, if $\mu$ is constant, then $(M,g)$ is an Einstein manifold with negative scalar curvature (see~\cite[p.~31, Thm.~8]{BFGK1991}).

    We claim that complete manifolds with generalized imaginary Killing spinors admit a Weyl structure reducible in the direction of the Lee form.
    In~\cite[Cor.~9]{F1990}, it was shown that if $q_\psi>0$, then $\mu$ is constant.
    Hence, according to~\cite[Thm.~2]{B1989*}, $\psi$ would be an imaginary Killing spinor of type~2.
    Then $(M,g)$ is isometric to the real hyperbolic space, see~\cite[Thm.~1]{B1989*}.
    However, the real hyperbolic space already has imaginary Killing spinors of type~$1$.
    Thus, we prove here the existence of a Weyl structure reducible in the direction of the Lee form only for generalized imaginary Killing spinors of type~$1$.
    
    Let us consider the $1$-form $\theta \coloneqq - 2\mu \, \xi ^\flat$ and the Weyl structure $\nabla$ on $(M,g)$ whose Lee form is $\theta$:
    \begin{equation}\label{eq:3.1}
    \nabla _X Y = \nabla^{g} _X Y -  \theta  (Y) \,X + g(X,Y)\,\theta^\sharp - \theta(X)\, Y.
    \end{equation}
    We differentiate the identity $i \psi = \xi \cdot \psi$ with respect to $\nabla^g _X$:
    \begin{align*}
	-\mu  X \cdot \psi
            &\;=\; \nabla^g _X( \xi \cdot \psi) \;=\; (\nabla^g _X \xi) \cdot \psi +\xi \cdot \nabla^g _X \psi \\
		&\;=\;  (\nabla^g _X \xi) \cdot \psi + \mu i\, \xi \cdot X \cdot  \psi\\
		&\;=\; (\nabla^g _X \xi) \cdot \psi -  \mu i\,X \cdot \xi \cdot \psi - 2 \mu i\, g(X,\xi)\,\psi	\\
		&\;=\; (\nabla^g _X \xi) \cdot \psi + \mu\, X \cdot \psi - 2\mu \, g(X,\xi)\, \xi \cdot \psi.
	\end{align*}
    Since $\psi\neq 0$, it follows that
    \begin{equation}\label{eq:3.2}
        \nabla^g _X \xi = -2\mu\, \big(X - g(X,\xi)\, \xi\big).
    \end{equation}
    Finally we compare~\eqref{eq:3.1} and~\eqref{eq:3.2} and obtain
    \begin{equation}\label{eq:3.3}
    \nabla _X \xi \;=\; -2\mu\, \big(X - g(X,\xi)\, \xi\big) + 2\mu\, X \;=\; 2\mu \, g(X,\xi)\, \xi.
    \end{equation}
    Let us now take $\mathcal{I} \coloneqq \mathcal{C}^\infty (M) \, \cdot \, \xi$, the distribution generated by the unit vector field $\xi$.
    Therefore~\eqref{eq:3.3} shows that $\mathcal{I}$ is $\nabla$-parallel, as claimed.
\end{example}

\begin{example}\label{ex:3.3}
Let us consider a quintuple $(M,g,\varphi,\xi,\eta)$ consisting of the following data:
an odd-dimensional Riemannian manifold $(M,g)$, a $(1,1)$–tensor field $\varphi$, a unit vector field $\xi$ (called the \emph{Reeb vector field}), and a $1$-form $\eta$ satisfying $\eta(\xi)=1$.
The quintuple $(M,g,\varphi,\xi,\eta)$ is called an \emph{almost contact metric manifold} if
\begin{equation*}
\varphi^{2}X = -X + \eta(X)\,\xi, \qquad
g\bigl(\varphi X,\varphi Y\bigr) = g(X,Y) - \eta(X)\,\eta(Y),
\end{equation*}
for all vector fields $X,Y$ on $M$.
In particular, these conditions imply $\varphi \xi = 0$ and $\eta \circ \varphi = 0$.
See~\cite{B2010} for further exposition of this topic.

\begin{definition}\label{def:3.4}
    Let $\alpha \in \mathcal{C}^{\infty}(M)$ be a non-zero function.
    An almost contact metric manifold $(M,g,\varphi,\xi,\eta)$ is called an \emph{$\alpha$-Kenmotsu manifold} if, for any vector fields $X$, $Y \in \mathfrak{X}(M)$, the following condition holds:
    \begin{equation}\label{eq:3.4}
        (\nabla^g _X \varphi )Y = -\alpha\big( g(  X,\varphi Y)\,\xi+\eta(Y)\,\varphi X\big).
    \end{equation}
    An $\alpha$-Kenmotsu manifold $(M,g,\varphi,\xi,\eta)$ is called \emph{Kenmotsu} if $\alpha = 1$.
    The Kenmotsu class of almost contact metric manifolds was first introduced in~\cite{K1972}.
\end{definition}

\begin{remark}
    According to the Chinea–González classification (\cite{CG1990}) of almost-contact metric structures, the $\alpha$-Kenmotsu class is described by the submodule $\mathcal{C}_5$.
    Setting $\alpha \equiv 0$ reduces the defining condition of an $\alpha$-Kenmotsu structure to that of a coKähler manifold.
    Namely, $\nabla^g g= 0$, $\nabla^g \varphi = 0$, and $\nabla^g \xi =0$.
    Therefore, requiring $\alpha$ to be non-zero is essential for obtaining an almost contact metric manifold in the strict class $\mathcal{C}_5$.
\end{remark}

We now construct a Weyl structure reducible in the direction of the Lee form for these manifolds.
Taking $Y=\xi$ in~\eqref{eq:3.4} gives $(\nabla^g _X \varphi) \xi = -\alpha \, \varphi X$; by definition we also have $(\nabla^g _X \varphi) \xi = - \varphi \big(\nabla^g _X  \xi\big)$.
Combining these identities with  $\varphi^{2}X = -X + \eta(X)\,\xi$ we obtain:
\begin{equation*}
\nabla^g _X \xi = \alpha\big(X-\eta(X)\,\xi\big) + \eta \big(\nabla^g _X \xi\big)\, \xi.
\end{equation*}
Finally, since $\xi$ is a unit vector field, we have $2\,\eta \big(\nabla^g _X \xi\big) = 2\,g(\nabla^g _X \xi, \xi) = X(g (\xi,\xi))  =0$.
This yields:
\[
\nabla^g _X \xi = \alpha\big(X-\eta(X)\,\xi\big).
\]

We set $\theta \coloneqq \alpha \eta$ and let $\nabla$ denote the Weyl structure whose Lee form is $\theta$ (see~\eqref{eq:2.1}).
Then we have
\begin{equation}\label{eq:3.5}
    \nabla _X \xi \;=\; - \alpha \, \eta(X)\, \xi \;=\; - \theta (X)\,\xi.
\end{equation}
This implies that $\nabla$ is reducible in the direction of the Lee form.
\end{example}
\section{Main Result}\label{sec:main}

The purpose of this section is to prove~\ref{thm:1.1}.
To do so, we begin with the following lemma, which involves the geometry of the universal cover.

\begin{lemma}\label{lem:4.1}
Let $(M,g)$ be a complete Riemannian manifold, and let $\nabla$ be a closed Weyl structure whose Lee form is $\theta$.
Suppose that $\nabla$ is reducible in the direction of the Lee form.
Then the universal cover $(\tilde{M},\tilde{g})$ is isometric to a warped product of the form $(\mathbb{R} \times N, dt^2 + e^{2f(t)}h_N )$ for some complete Riemannian manifold $(N,h_N)$, where $\tilde{\theta} = df$.
\end{lemma}

\begin{proof}
Let $(\tilde{M}, \tilde{g})$ be the universal cover of $(M,g)$.
By hypothesis, there exists a $\nabla$-parallel distribution $\mathcal{I}$ of dimension $1$ such that $\theta^\sharp \in \mathcal{I}$.
We denote by $\tilde{\nabla}$, $\tilde{\theta}$, and $\tilde{\mathcal{I}}$ the lifts to the universal cover $(\tilde{M},\tilde{g})$ of $\nabla$, $\theta$, and $\mathcal{I}$, respectively.

Since $\theta$ is closed, it follows that $\tilde{\theta}$ is also closed.
Moreover, because $\tilde{M}$ is simply connected, $\tilde{\theta}$ must be exact.
Thus, there exists $f\in \mathcal{C}^\infty(\tilde{M})$ such that $\tilde{\theta}=df$.
Therefore, $(\tilde{M}, h := e^{-2f}\tilde{g})$ is a Riemannian manifold whose Levi-Civita connection is $\tilde{\nabla}$.

Because $\mathcal{I}$ is $\nabla$-parallel, it follows that $\tilde{\mathcal{I}}$ is $\tilde{\nabla}$-parallel.
Since $\tilde{M}$ is simply connected, the rank-$1$ $\tilde{\nabla}$-parallel subbundle defined by $\tilde{\mathcal{I}}$ is therefore trivial.
Hence, there exists a global non-vanishing $\tilde{\nabla}$-parallel section; we denote it by $\zeta \in \mathfrak{X}(\tilde{M})$ and $\beta \coloneqq \zeta ^{\flat_{h} }\in \Omega^1 (\tilde{M})$, where $(\cdot)^{\flat_{h}}$ denotes the musical isomorphism induced by $h$ on $\tilde{M}$.
Thus, $\| \beta \| _{h}$ is constant and, in particular, $\beta$ never vanishes.
From now on, without loss of generality, we assume that $\| \beta\|_{h} = 1$; otherwise, we normalize $\beta$.

Consider the distribution $\mathcal{D} = \{X \in \mathfrak{X}(M) :\: g(X,Y) = 0,\, \forall Y\in \mathcal{I}\}$.
Recall that $\mathcal{I}$ and $\mathcal{D}$ are two integrable, $\nabla$-parallel, and complementary distributions (see Section~\ref{sec:prelim}).
Let us now consider $\tilde{\mathcal{D}}$ and $\tilde{\mathcal{I}}$, the lifts of $\mathcal{D}$ and $\mathcal{I}$, respectively.
It is straightforward that $\tilde{\mathcal{I}} = \mathcal{C}^{\infty}(\tilde{M}) \,\zeta$.
Furthermore, since $\tilde{\nabla} h = 0$ and $\tilde{\nabla} \beta = 0$, we have $\tilde{\nabla} \zeta = 0$.
Thus, $\tilde{\mathcal{I}}$ is $\tilde{\nabla}$-parallel and integrable.

Fix a point $p \in M$ and choose a lift $\tilde{p} \in \tilde{M}$.
We claim that the maximal integral leaf of $\mathcal{D}$ through $p$ is $\nabla$-geodesically complete.
To prove this, it suffices to show that the connected maximal integral leaf $N$ of $\tilde{\mathcal{D}}$ through $\tilde{p}$ is $\tilde{\nabla}$-geodesically complete.

Since $df \wedge \beta = 0$, contracting with any $X\in \tilde{\mathcal{D}}$ yields $df (X) = 0$.
We set $C\coloneqq f(\tilde{p})$.
Since $f$ is continuous, the set $f^{-1} (C)$ is closed and contains $N$.
Furthermore, by the maximality condition, $N$ is the connected component of $f^{-1} (C)$ containing $\tilde{p}$.
Therefore, $N$ is a closed submanifold.

Because $(\tilde{M}, \tilde{g})$ is complete and $N$ is closed and without boundary, we conclude that $(N, \tilde{g}|_{N})$ is complete (see~\cite[Cor.~2.10]{C1992}).
It is globally homothetic to $(N, e^{-2C}\tilde{g}|_{N})$, and therefore it is also complete.
Indeed, the argument is independent of the choice of $p$.
Thus, all the integral leaves of $\tilde{\mathcal{D}}$ are $\tilde{\nabla}$-geodesically complete.

In~\cite{PG1993}, the authors study the conditions under which a simply connected Riemannian manifold with reducible holonomy can be a twisted product of two Riemannian manifolds.
In particular, the following theorem is a consequence of~\cite{PG1993}.

\begin{theorem}[\cite{PG1993}]\label{thm:PG1993}
    Let $(\tilde{M},h)$ be a simply connected Riemannian manifold such that the tangent bundle $T\tilde{M}$ admits two complementary, integrable, and parallel distributions $\tilde{\mathcal{I}}$ and $\tilde{\mathcal{D}}$.
    Suppose that the leaves of $\tilde{\mathcal{D}}$ are complete.
    Then $\tilde{M}$ is globally isometric to a product of two Riemannian manifolds $L \times N$, where $L$ and $N$ are the maximal integral leaves of $\tilde{\mathcal{I}}$ and $\tilde{\mathcal{D}}$ through any $\tilde{p}\in \tilde{M}$, respectively.
\end{theorem}

Since $(\tilde{M}, h = e^{-2f} \tilde{g})$, $\tilde{\mathcal{I}}$ and $\tilde{\mathcal{D}}$ satisfy the conditions of~\ref{thm:PG1993}, it follows that $(\tilde{M}, h)$ is globally isometric to the product of two Riemannian manifolds $(L, \beta \otimes \beta) \times (N, h_N \coloneqq h|_N)$.
Here, $L$ is a connected integral leaf of $\tilde{\mathcal{I}}$, which is one-dimensional.
Moreover, because $(\tilde{M} , \tilde{g})$ is complete, it follows that $(\tilde{M},\tilde{g})$ is isometric to $(\mathbb{R} \times N, dt^2 + e^{2f(t)} h_N)$, where $dt = e^{f} \beta$.
\end{proof}

We now explore when the deck transformations are isometries for $h$.

Let $\gamma \in \pi _1 (M)$ be a closed curve in $M$.
Any deck transformation $T_{\gamma}\colon \tilde{M} \longrightarrow \tilde{M}$ is an isometry of $\tilde{g}$ and preserves $\tilde{\theta}$, $\tilde{\mathcal{I}}$ and $\tilde{\mathcal{D}}$,  
\begin{equation}\label{eq:4.1}
    (T_{\gamma})^* \tilde{g} = \tilde{g}, \quad (T_{\gamma})^* \tilde{\theta} = \tilde{\theta}, \quad (T_{\gamma})_* \tilde{\mathcal{I}} = \tilde{\mathcal{I}}, \quad (T_{\gamma})_* \tilde{\mathcal{D}} = \tilde{\mathcal{D}}.
\end{equation} 
From the last two identities in~\eqref{eq:4.1}, since $(\tilde{M}, \tilde{g})$ is a warped product $(\mathbb{R} \times N, dt^2 + e^{2f(t)} h_N)$, there exist two functions $\tau_\gamma \colon \mathbb{R} \longrightarrow \mathbb{R}$ and $\varphi_\gamma \colon N \longrightarrow N$, such that
\[
T_{\gamma} (r, \tilde{q}) = (\tau_\gamma (r), \varphi_{\gamma} (\tilde{q})), \quad \text{ for all } (r,\tilde{q}) \in \mathbb{R} \times N.
\]
Since $ (T_{\gamma})^* \tilde{\mathcal{I}} = \tilde{\mathcal{I}}$ and $(T_{\gamma})^* \tilde{g} = \tilde{g}$, it follows that $\tau_\gamma$ is an isometry of $(\mathbb{R}, dt^2)$.

Since $\tilde{\theta} \wedge dt = 0$, there exists $\alpha \in \mathcal{C}^\infty(\tilde{M})$ such that $\tilde{\theta} = \alpha\, dt$.
Because $\tilde{\theta}$ is closed, $\alpha$ is constant along the leaves of $\tilde{\mathcal{D}}$.
We use that $\tau_\gamma$ is an isometry of $(\mathbb{R}, dt^2)$ to obtain that there are exactly two possibilities:
\begin{itemize}
    \item[(a)] $\tau_\gamma ^* dt = -dt$, which implies $\tau_\gamma ^* \alpha = -\alpha$,
    \item[(b)] $\tau_\gamma ^* dt = dt$, which implies $\tau_\gamma ^* \alpha = \alpha$.
\end{itemize}
Case (a).~Then $\tau_\gamma(r) = - r + C_\gamma$, where $C_\gamma \in \mathbb{R}$ is constant.
Using $(T_{\gamma})^* \tilde{g} = \tilde{g}$, we obtain
\[
 dt^2 + e^{2f(t)} h_N \;=\; \tilde{g}  \;=\; (T_{\gamma})^* \tilde{g} \;=\; dt^2 + e^{2f(-t + C_\gamma)}  (\varphi_\gamma)^* h_N.
\]
Thus, 
\[
(\varphi_\gamma)^* h_N = e^{2(f(t) - f(-t + C_\gamma))} h_N .
\]
Now we apply the Fundamental Theorem of Calculus and $df = \alpha \, dt$ to obtain:
\[
 (\varphi_\gamma)^* h_N  = e^{2 \int _{-t + C_\gamma}  ^t \alpha (s) ds} h_N.
\]
Finally, we use the identity $\tau_\gamma^* \alpha = -\alpha$.
Thus, for all $t \in \mathbb{R}$, we have $\alpha (-t + C_\gamma) = - \alpha(t)$; hence $\int _{-t + C_\gamma}  ^t \alpha (s) ds = 0$ and
\begin{equation*}
    (\varphi_\gamma)^* h_N  = h_N.
\end{equation*}
Therefore, in the case (a), $ (T_\gamma)^* h = h$.

Case (b).~Then $\tau_\gamma(r) = r + C_\gamma$, where $C_\gamma$ is a constant.
Using again $(T_{\gamma})^* \tilde{g} = \tilde{g}$, we obtain
\[
 dt^2 + e^{2f(t)} h_N \;=\; \tilde{g}  \;=\; (T_{\gamma})^* \tilde{g} \;=\; dt^2 + e^{2f(t + C_\gamma)}  (\varphi_\gamma)^* h_N .
\]
Since we cannot at this stage ensure that $f$ is periodic, we only have
\[
 (\varphi_\gamma)^* h_N  = e^{-2f(t + C_\gamma) + 2f(t)} h_N.
\]
Now we apply the Fundamental Theorem of Calculus and $df = \alpha \, dt$ to obtain:
\[
 (\varphi_\gamma)^* h_N  = e^{-2 \int _t ^{t + C_\gamma}\alpha (s) ds} h_N.
\]
Finally, we use the identity $\tau_\gamma^* \alpha = \alpha$.
Thus, for all $t \in \mathbb{R}$, we have $\alpha (t + C_\gamma) = \alpha(t)$; consequently, 
\begin{equation}\label{eq:4.2}
    (\varphi_\gamma)^* h_N  = e^{-2 \kappa_\gamma} h_N
\end{equation}
where $\kappa_\gamma= \int _t ^{t + C_\gamma}\alpha (s)\, ds$ is constant and independent of $t$.

\begin{remark}\label{rmk:4.3}
Note that:
    \begin{enumerate}
        \item If $C_\gamma = 0$, then $\kappa_\gamma = 0$ and, consequently, $(T_\gamma)^* h = h$.
        \item The Weyl structure $\nabla$ is non-exact if and only if there exists a curve $\gamma \in \pi_1(M)$ such that  $\kappa_{\gamma} \neq 0$, $(\tau _\gamma)^* dt = dt$ and $C_\gamma \neq 0$.
    \end{enumerate}
\end{remark}
  
\subsection{The compact case} In~\cite{K2019}, the author studies similarity structures on compact manifolds.
Since their universal covers admit Riemannian metrics, several holonomy conditions are established in~\cite{K2019}.
In particular, the following theorem is addressed.

    \begin{theorem}[{\cite[Thm.~1.5]{K2019}}]\label{thm:K2019}
    Let $(M,g)$ be a compact Riemannian manifold, and let $\nabla$ be a closed Weyl structure whose Lee form is $\theta$.
    Let $(\tilde{M}, \tilde{g})$ be the universal cover of $(M,g)$ and let $\tilde{\theta}$ and $\tilde{\nabla}$ denote the lifts of $\theta$ and $\nabla$, respectively.
    Since $\tilde{\theta}$ is closed and $\tilde{M}$ is simply connected, it follows that $\tilde{\theta} = df$ and $\tilde{\nabla}$ is the Levi-Civita of $e^{-2f}\tilde{g}$.
    Then one of the following four cases must occur:
	\begin{enumerate}
		\item The 1-form $\theta$ is exact.
		\item $(\tilde{M}, e^{-2f}\tilde{g})$ is flat.
		\item $(\tilde{M}, e^{-2f}\tilde{g})$ is irreducible.
		\item $(\tilde{M}, e^{-2f}\tilde{g})$ is isometric to $\mathbb{R} ^q \times N$, with $q > 0$ and where $N$ is non-flat, non-complete and with irreducible holonomy.
	\end{enumerate}
\end{theorem}

\begin{theorem}\label{thm:4.5}
    Let $(M,g)$ be a compact Riemannian manifold, and let $\nabla$ be a closed Weyl structure whose Lee form is $\theta$.
    Suppose that $\nabla$ is reducible in the direction of the Lee form.
    Then $\nabla$ is flat or exact.
\end{theorem}
\begin{proof}
     Using~\ref{lem:4.1} and~\ref{thm:K2019}, we have that $(\tilde{M}, h \coloneqq e^{-2f}\tilde{g})$ decomposes into two Riemannian manifolds.
     Thus, Case~(3) is not possible.
     Furthermore, since $L$ is flat and $N$ is complete, it follows that Case~(4) cannot occur.
     Consequently, the Weyl structure $\nabla$ is flat or exact.
\end{proof}

Note that~\ref{thm:1.1} generalizes~\ref{thm:4.5} without the compactness assumption.
Therefore, we are now ready to prove it.

\noindent \textbf{Proof of~\ref{thm:1.1}.}
    It remains to treat the non-exact case.
    We keep the notation of~\ref{lem:4.1} and~\ref{rmk:4.3}.
    Based on the discussion above~\ref{rmk:4.3}, we conclude that $\nabla$ is non-exact if and only if there exists a closed curve $\gamma \in \pi_1(M)$ such that its deck transformation $T_\gamma \colon L \times N \longrightarrow L \times N$ satisfies
    \[
        T_{\gamma} (r, \tilde{q}) = (r + C_{\gamma}, \varphi_{\gamma} (\tilde{q})), \quad \text{ for all } (r,\tilde{q}) \in \mathbb{R} \times N,
    \]
    where $C_\gamma$ is a constant that satisfies $\alpha (t + C_\gamma) = \alpha (t)$, and $\kappa_\gamma = \int _t ^{t + C_\gamma}\alpha (s)\, ds$ is a non-zero constant.

    From~\eqref{eq:4.2}, we find that $\varphi_\gamma$ is a similarity of $(N,h_N)$ of ratio $e^{-\kappa_\gamma} >0$.
    Without loss of generality, we assume that $\kappa_\gamma > 0$; otherwise, we consider $\varphi_\gamma ^{-1}$ instead of $\varphi_\gamma$.
    
    Let $d_N$ denote the length-minimizing distance in $(N, h_N)$.
    Since $(N, h_N)$ is a complete connected Riemannian manifold, the metric space $(N,d_N)$ is also complete.
    Note that $\varphi_\gamma$ is a similarity of ratio $e^{-\kappa_\gamma} < 1$ and therefore it is a contraction of the metric space $(N,d_N)$.
    By the Banach Fixed-Point Theorem, there exists a unique fixed point $q\in N$ for $\varphi_\gamma$.
    Furthermore, for every $x_0\in N$, the sequence $\{x_k \coloneqq\varphi_\gamma ^k (x_0)\}_{k \in \mathbb{N}}$ converges to $q$.

    This last part of the proof is inspired by~\cite[Prop.~3.5]{K2019}.
    Remember that $\tilde \nabla$ is the Levi-Civita connection on $(\tilde{M},h)$.
    Furthermore, $(\tilde{M}, h)$ is isometric to a product of two Riemannian manifolds $(\mathbb{R}, dt^2) \times (N,h_N)$.
    Thus, $\tilde{\nabla}$ is flat if and only if $(N,h_N)$ is flat.
    
    Let $\nabla^{h_N}$ be the Levi-Civita connection on $(N,h_N)$ and let $R_p$ denote the $(0,4)$-curvature tensor of $\nabla ^{h_N}$ at $p \in N$.
    Since $\varphi_\gamma \in \mathrm{Aff}(N,\nabla ^{h_N})$, it follows that $\varphi_\gamma ^* R = e^{-2\kappa_\gamma} R$.
    In particular,
    \[
     \left\|R_{p} \right\|^2 \;=\; e^{4\kappa_\gamma r}\left\| ((\varphi_\gamma ^r)^* R)_{p} \right\|^2 \;=\; e^{4 \kappa_\gamma r} e^{-8 \kappa_\gamma r} \left\|R_{\varphi_\gamma ^r(p)} \right\|^2
    \]
    where $\| R_p\|^2 = \sum_{i,j,k,l=1} ^{n-1} \! R_p (e_i,e_j,e_k,e_l)^2$, and $\{e_1, \dots, e_{n-1}\}$ is an orthonormal basis in $T_pN$, where $n = \dim M$.
    Taking the limit as $r \rightarrow \infty$, we obtain that $R_p = 0$ for every $p \in N$.
    Thus $\nabla^{h_N}$ is flat.\hfill$\blacksquare$

    We now prove that if $\nabla$ is flat, then $(N,h_N)$ is isometric to $\mathbb{R}^{n-1}$, where $n = \dim M$.
    Let $B(0,\lambda) \subset \mathbb{R} ^{n-1}$ and $B(q,\lambda) \subset N$ be the open balls centered at $0$ and $q$, respectively, each of radius $\lambda >0$.
    Because $R$ is flat, there exists an $\varepsilon >0$ such that the exponential map $\exp _q : B(0,\varepsilon)\longrightarrow B(q, \varepsilon)$ is an isometry.
    Let $\phi$ be the homothety of $\mathbb{R}^{n-1}$ centered at $0$ of ratio $e^{-\kappa_\gamma} >0$.
    Then, for each $r\in \mathbb{N}$, we set 
    \begin{align*}
        \Phi_r \colon B(0, e^{\kappa_\gamma r} \varepsilon) &\longrightarrow B(q,e^{\kappa_\gamma r} \varepsilon) \\
        x &\longmapsto  \left(\varphi_\gamma ^{-r}\circ\exp_q \circ \,\phi^{r} \right)(x).
    \end{align*}
    Since $\Phi_r$ is the composition of two inverse-ratio similarities and one isometry, it follows that $\Phi_r$ is an isometry.
    Furthermore, the differentials $d\Phi_r$ at $q$ coincide for all $r \in \mathbb{N}$.
    Thus, for each $r_0\in \mathbb{N}$, if $x \in B(0, e^{\kappa_\gamma r_0} \varepsilon)$, then $\Phi_{r_0} (x) = \Phi_r (x)$ for all $r\geq r_0$.
    Define $\Phi \colon \mathbb{R}^{n-1} \longrightarrow N$ by $\Phi(x) = \Phi_{r(x)} (x)$, where  $r(x)$ is the first natural number for which $\phi^{r(x)} (x) \in B(0,\varepsilon)$.
    This map $\Phi$ is a global isometry; hence $(N,h_N)$ is isometric to the Euclidean space.

From the proof of~\ref{thm:1.1}, together with the subsequent paragraph, we deduce the following result, which will be used in Section~\ref{sec:apps}.

\begin{proposition}\label{lem:4.7}
    Let $(N, h_N)$ be a complete connected Riemannian manifold of dimension $m$ and let $\varphi$ be a similarity of $(N,h_N)$ of ratio $\lambda < 1$.
    Then $N$ is isometric to $\mathbb{R}^m$.
\end{proposition}

We now study a simple criterion for exactness.
If $(\tilde M,h)$ is complete, then $\nabla$ is exact.
To see this, note that any deck transformation belongs to $\mathrm{Aff}(\tilde{M}, \tilde{\nabla}) = \mathrm{Sim}(\tilde{M},h)$.
Suppose $\nabla$ were non-exact; then there would exist a deck transformation $T$ which is a similarity of ratio $\lambda \neq 1$.
By the Banach Fixed-Point Theorem, either $T$ or $T^{-1}$ must have a fixed point, contradicting the fact that non-trivial deck transformations act freely.
This contradiction shows $\nabla$ is exact.

\section{Applications}\label{sec:apps}

In this section, we study complete Kenmotsu manifolds.
Let $\alpha$ be a non-zero function and $(M,g,\varphi,\xi, \eta)$ be an $\alpha$-Kenmotsu manifold (see~\ref{ex:3.3} for the definition).

Recall that the Weyl structure $\nabla$ in~\ref{ex:3.3}, whose Lee form is $\theta = \alpha \, \eta$, is reducible in the direction of the Lee form.
Furthermore, the Lee form is closed if and only if
\begin{equation}\label{eq:5.1}
    X(\alpha) = 0, \quad  \forall X\in \{A \in \mathfrak{X}(M):\: g(A,\xi) = 0 \}.
\end{equation}
In particular, when $\alpha\equiv1$ (the classical Kenmotsu case), condition~\eqref{eq:5.1} holds automatically.
Hence, any Weyl structure $\nabla$ from~\ref{ex:3.3} that satisfies~\eqref{eq:5.1} satisfies the hypotheses of~\ref{thm:1.1}.

\begin{corollary}\label{cor:5.1}
    Let $(M,g,\varphi,\xi, \eta)$ be a complete $\alpha$-Kenmotsu manifold such that~\eqref{eq:5.1} holds.
    Then the following statements hold:
    \begin{enumerate}
        \item If $\nabla$ is exact, then $(M,g,\varphi,\xi, \eta)$ is globally conformal to a coKähler manifold (see~\ref{rmk:5.2}).
        \item If $\nabla$ is non-exact, then $(M,g)$ (with $\dim M = 2m +1$) is isometric to a proper quotient of the warped product $(\mathbb{R}, dt^2) \times_f (\mathbb{R} ^{2m}, h, J)$, where $(\mathbb{R}^{2m}, h, J)$ is the standard Kähler structure, and $f\in \mathcal{C}^{\infty}(\mathbb{R})$.
    \end{enumerate}
\end{corollary}

\begin{remark}\label{rmk:5.2}
    Recall that an almost contact metric manifold $(M,g,\varphi,\xi, \eta)$ is coKähler if and only if
    \[
    \nabla^g \varphi = 0, \quad \nabla^g \xi = 0, \quad \nabla^g \eta = 0.
    \]
\end{remark}

\begin{proof}
    If $\nabla$ is exact, there exists $f\in \mathcal{C}^{\infty}(M)$ such that $\theta = df$.
    Consequently, $g$ is conformally equivalent to $h := e^{-2f}g$, whose Levi-Civita connection is $\nabla$.
    First, we verify that $(M,h,\varphi,e^{f}\xi,e^{-f}\eta)$ is an almost contact metric manifold: 
    \[
    \varphi ^2 X \;=\; -X + \eta (X) \xi \;=\; - X + ( e^{-f} \eta)(X)\, e^{f}\xi
    \]
    and 
    \begin{align*}
    h (\varphi X, \varphi Y) \;=\; e^{-2f}g(\varphi X, \varphi Y) &\;=\; e^{-2f} \big(g(X, Y) - \eta(X) \eta(Y)\big) \\
    &\;=\; h(X, Y) - (e^{-f} \eta)(X) ( e^{-f} \eta)(Y)
    \end{align*}
    for all $X$, $Y\in \mathfrak{X}(M)$.
    Next, we show that the structure is coKähler.
    Since $\nabla \xi = -\theta \otimes \xi$ (see~\eqref{eq:3.5}), we obtain $\nabla (e^{f}\xi) = 0$.
    It remains to show that $\nabla \varphi = 0$:
    \begin{align*}
        (\nabla_X\varphi)(Y)
            &\;=\; \nabla_X\bigl(\varphi Y\bigr)-\varphi(\nabla_XY) \\[2pt]
            &\;=\; \left(\nabla^{g}_{X} \varphi\right) (Y)
                -\underbrace{\theta\bigl(\varphi Y\bigr)}_{=0}\,X
                + g\bigl(X,\varphi Y\bigr)\,\theta^{\sharp}
                -\cancel{\theta(X)\,\varphi Y} \\[2pt]
            &\hspace{2.5cm} +\theta(Y)\,\varphi X
                - g(X,Y)\,\underbrace{\varphi \theta^{\sharp}}_{=0}
                +\cancel{\theta(X)\,\varphi Y}
    \end{align*}
    Finally, using~\eqref{eq:3.4}, we obtain that $\left(\nabla^{g}_{X} \varphi\right) (Y) = -\theta(Y)\,\varphi X - g\bigl(X,\varphi Y\bigr)\,\theta^{\sharp}$.
    Consequently, $(M, h, \varphi, e^{f}\xi, e^{-f} \eta)$ is a coKähler manifold.

    If $\nabla$ is non-exact, then, by~\ref{thm:1.1}, $\nabla$ is flat.
    Consequently, by~\ref{lem:4.1}, the universal cover $(\tilde{M},\tilde{g})$ of $(M,g)$ is globally isometric to $(\mathbb{R} \times \mathbb{R}^{2m}, dt^2 + e^{2f(t)}\,h)$, where $h$ is the Euclidean metric of $\mathbb{R}^{2m}$.
    Finally, we observe that
    \[
        (\nabla^g _X \varphi )Y = -\alpha\big( g(X,\varphi Y)\,\xi+\eta(Y)\,\varphi X\big)
    \]
    vanishes identically if we choose $X$, $Y\in \mathcal{D} = \{ X\in \mathfrak{X}(M):\: g(X, \xi) = 0 \}$.
    Thus, $J \coloneqq \tilde{\varphi}|_{\mathbb{R}^{2m}}$ is compatible with the metric $h$ and parallel with respect to the Euclidean metric of $(\mathbb{R} ^{2m}, h)$.
    Since every Kähler structure on $\mathbb{R}^{2m}$ compatible with the Euclidean metric is isometric to the standard Kähler structure of $\mathbb{R}^{2m}$, the proof is complete.
\end{proof}

Therefore, \ref{cor:5.1} yields the following classification of complete Kenmotsu manifolds:

\begin{corollary}
    Every complete Kenmotsu manifold is globally conformal to a coKähler manifold or is locally conformally flat.
\end{corollary}

\subsection{Homogeneous Kenmotsu Manifolds}
We now turn our attention to the homogeneous setting.
We say that an $\alpha$-Kenmotsu manifold $(M,g,\varphi,\xi, \eta)$ is homogeneous when a Lie group $G \subset \mathrm{Isom}(M,g)$ acts transitively and preserves $\varphi$.
Note that, since $\varphi(\xi)= 0$ and $\eta (\,\cdot\,) = g(\,\cdot\,, \xi)$, preserving $\varphi$ implies preserving $\xi$ and $\eta$.
For homogeneous Kenmotsu manifolds, with $\alpha \equiv 1$, \cite{L2023} proved the following:

\begin{theorem}[\cite{L2023}]\label{thm:L2023} 
    Every homogeneous Kenmotsu manifold is a contractible space.
\end{theorem}

\ref{thm:L2023} shows that the topology of a homogeneous Kenmotsu manifold is severely restricted: the space must be contractible.
Our next result shows that, once this topological constraint is combined with the universal cover description (see~\ref{lem:4.1}), the geometry becomes completely rigid.
Specifically, the function $f$ in~\ref{lem:4.1} must be linear: $f(t)=\alpha\,t$ with constant $\alpha$.
Consequently, the underlying Riemannian manifold is isometric to $(\mathbb{R} \times \mathbb{R}^{n-1}, g_{(t,x)} = dt^2 + e^{2 \alpha \,t} g_{\mathrm{eucl}})$.
To connect this with the standard model, we take the upper half-plane model of $\mathbb{H}(-\alpha^2)$, whose hyperbolic metric is $\frac{dr^2 + dx^2 }{\alpha^2 r^2}$ in coordinates $(r, x)\in \mathbb{R}^+ \times\mathbb{R}^{n-1}$.
Setting $t = - \log (r)$ yields the metric $dt^2 + e^{2 \alpha \,t} |dx|^2$.
This well-known warped product description of the real hyperbolic space complements the homogeneous model presented in~\cite[Sec.~3.1]{CGS2009}.

To conclude, we extend the argument to general closed Weyl structures that are reducible in the direction of the Lee form, not just to the Kenmotsu case.

\begin{theorem}\label{thm:5.5}
    Let $(M,g)$ be a Riemannian manifold and $\nabla$ a closed Weyl structure reducible in the direction of the Lee form $\theta$ with a codimension $1$ distribution $\mathcal{D}$, see~\ref{def:3.1}.
    Suppose a Lie group $G\subset \mathrm{Isom}(M,g)$ acts transitively on $M$ and preserves $\mathcal{D}$.
    Then $(M,g)$ is isometric to the real hyperbolic space with constant sectional curvature $-\|\theta\|^2$.
\end{theorem}

\begin{proof}
    Since $(M,g)$ is homogeneous, it is complete.
    Consequently, the Weyl structure $\nabla$ satisfies the hypotheses of~\ref{thm:1.1}; hence it is flat or exact.

    Our proof proceeds in two steps:
    \begin{itemize}
        \item[(1)] If $(M,g)$ is simply connected, then $(M,g)$ is isometric to the real hyperbolic space $\mathbb{H}(- \|\theta\|^2)$, whose sectional curvature equals $- \|\theta\|^2$.
        \item[(2)] Show that $(M,g)$ is simply connected.
        \item[] 
    \end{itemize}
    Note that~\ref{thm:L2023} already covers Step~2 in the Kenmotsu case.

    Step 1.~Since $M$ is simply connected and $\nabla$ is closed, we have that $\nabla$ is exact; that is, $\nabla$ is the Levi-Civita connection of $e^{-2f}g$.
    \ref{lem:4.1} states that $(M,g)$ is isometric to $(\mathbb{R} \times N, g = dt^2 + e^{2f}h)$ for some complete manifold $(N,h)$, where $df =\theta$.
    Recall that $N$ is an integral leaf of $\mathcal{D}$, see~\ref{thm:PG1993}.
    In particular,~\ref{lem:4.1} ensures the existence of a globally defined unit $1$-form $dt$ proportional to $\theta$.
    Thus, we can write $\theta = \alpha\,dt$, where $\alpha \coloneqq g(\theta^\sharp, (dt)^\sharp) \in \mathcal{C}^{\infty}(M)$.
    Note that $\alpha$ equals either  $- \| \theta \|$ or $\| \theta \|$, and $\alpha$ is not necessarily constant.
    
    Let $G \subset \mathrm{Isom}(M,g)$ be a Lie group acting transitively on $M$ and preserving $\mathcal{D}$.
    Let $t$, $s\in \mathbb{R}$ and $x \in N$.
    Using the transitivity condition, there exists $\Gamma \in G$ such that $\Gamma(t,x) = (s,x)$.
    Next, because $G$ preserves $\mathcal{D}$, we have $\Gamma^*dt$ equals either $dt$ or $-dt$, and $\Gamma_* \mathcal{D} = \mathcal{D}$.
    Then since $\Gamma$ is an isometry, we obtain
    \[
    dt^2 _{(t,x)} + e^{2f(t)} h_{x} \; =\; g_{(t,x)} = (\Gamma^*g)_{(t,x)} \;=\; dt^2_{(t,x)} + e^{2f(s)} (\Gamma^*h) _{x}.
    \]
    This implies that $(\Gamma^* h)_x= e^{2(f(t)-f(s))} h_x$,
    which means that $\varphi_{(t,s)} \coloneqq \Gamma|_N$ is a similarity of $(N,h)$ of ratio $e^{2(f(t)-f(s))}$.
    Since the Lee form is non-zero (see~\ref{def:3.1}), $f$ cannot be constant.
    Thus, there exist $\hat{r}$, $\hat{s} \in \mathbb{R}$ with $f(\hat{r}) \neq f(\hat{s})$; hence either $\varphi_{(\hat{r},\hat{s})}$ or $\varphi_{(\hat{r},\hat{s})} ^{-1}$ satisfies the conditions of~\ref{lem:4.7} and $(N,h)$ is isometric to $\mathbb{R} ^{n-1}$ with $n = \dim M$.
    Therefore, $(M,g)$ is isometric to
    \[
    (\mathbb{R} \times \mathbb{R}^{n-1}, g_{(t,x)} = dt^2 + e^{2 f(t)} g_{\mathrm{eucl}}).
    \]
    We now compute the Ricci and scalar curvatures of $(M,g)$ to derive the restrictions on $\alpha$ imposed by homogeneity.
    We claim that $\alpha$ is constant and non-zero.
    Since the Lee form is non-zero, $\alpha$ cannot be identically zero.
    Let us take a global coordinate system $(t, x_1 , \ldots x_{n-1})$.
    Following~\cite[p.~211, Cor.~43]{O1983}, we obtain
    \[
    \mathrm{Ric}_{tt} = - (n-1) \frac{\ddot{s}}{s}, \quad \mathrm{Ric}_{ij} = - \left( \frac{\ddot{s}}{s} + (n-2)\left(\frac{\dot{s}}{s}\right) ^2 \right) \delta_{ij},\quad  \forall i,\,j \in \{ 1, \ldots, n-1\},
    \]
    where $s(t) = e^{f(t)}$.
    Hence, the scalar curvature is constant on $M$ and equals
    \[
    \mathrm{Scal}(M,g) = - (n-1) \left( 2\,\frac{\ddot{s}}{s} + (n-2)\left(\frac{\dot{s}}{s}\right) ^2 \right).
    \]
    Substituting $s(t) = e^{f(t)}$ and $df = \alpha \,dt$ in the last formula, we obtain 
    \begin{equation}\label{eq:5.2}
    \mathrm{Scal}(M,g) = - (n-1) \left(  2\dot{\alpha} +  n \alpha^2 \right).
    \end{equation}
    Since $dt$ is globally defined and $\Gamma^* dt = \pm dt$ for all $\Gamma \in G$, it follows that $\mathrm{Ric}_{tt}$ is constant on $M$; hence
    \begin{equation}\label{eq:5.3}
        \mathrm{Ric}_{tt} = - (n-1) \left(  \dot{\alpha} +   \alpha^2 \right).
    \end{equation}
    When $n \geq 3$, comparing~\eqref{eq:5.2} with~\eqref{eq:5.3} shows that $\alpha^2$ is constant on $M$ and hence $\alpha$ is also constant.
    For $n = 2$, we may write $(M,g) \cong (\mathbb{R}^2,\; g = dt^2 + s(t)^2\,dx^2)$, with  $s(t)=e^{f(t)}$ and  $\alpha(t)=\tfrac{\dot s(t)}{s(t)}$.
    Let $E_{(t,x)} \coloneq \frac{1}{s(t)} \partial_x$ be the globally defined and unit vector field spanning the distribution $\mathcal{D}$.
    A direct computation (see \cite[Ch.~7, Prop.~35]{O1983}) gives
    \begin{equation}\label{eq:shape}
        g(\nabla^g_E\partial_t,E)\; =\;\frac{\dot s}{s}\;=\;\alpha.
    \end{equation}
    Since $G$ preserves $\mathcal D$ and the metric, for every $\Gamma\in G$ we have $\Gamma_*E=\pm E$ and $\Gamma_*\partial_t=\pm\partial_t$; thus, by \eqref{eq:shape},
    \[
        \alpha\circ\Gamma \;=\;g(\nabla^g_{\Gamma_*E}\Gamma_*\partial_t,\Gamma_*E)\;=\;\pm g(\nabla^g_E\partial_t,E)
        \;=\;\pm\alpha,
    \]
    so $|\alpha|$ is $G$–invariant and therefore constant on $M$.
    Thus $\alpha$ is constant on $M$.
    Hence, independently of the dimension, $(M,g)$ is isometric to the warped product manifold
    \[
    (\mathbb{R} \times \mathbb{R}^{n-1}, g_{(t,x)} = dt^2 + e^{2 \alpha \,t} g_{\mathrm{eucl}}),
    \]
    which is the real hyperbolic space of sectional curvature $-\alpha^2$.
    
    Step 2.~Recall that every Riemannian homogeneous manifold is locally isometric to its universal cover; thus $(M,g)$ is locally isometric to the real hyperbolic space.
    In particular, $(M,g)$ has negative sectional curvature and negative definite Ricci curvature.
    Therefore, $(M,g)$ satisfies the hypotheses of the following result:
    \begin{theorem}[\cite{K1962}] \label{thm:K1962}
        Let $(M,g)$ be a homogeneous Riemannian manifold with non-positive sectional curvature and negative definite Ricci tensor.
        Then $M$ is simply connected.
    \end{theorem}
    \ref{thm:K1962} of S.~Kobayashi therefore forces $(M,g)$ to be simply connected.
    Hence, $\nabla$ must be exact, so $(M,g)$ is globally isometric to the real hyperbolic space.
\end{proof}

In the $\alpha$-Kenmotsu setting we obtain:

\begin{corollary}\label{cor:5.7}
    Every homogeneous $\alpha$-Kenmotsu manifold is isometric to the real hyperbolic space with constant sectional curvature $- \alpha^2$.
\end{corollary}

\begin{proof}
    Let $(M,g,\varphi,\xi, \eta)$ be a homogeneous $\alpha$-Kenmotsu manifold and let $G$ be a Lie group acting transitively on $M$ and preserving $\varphi$, $\xi$, and $\eta$.
    From the definition of an $\alpha$-Kenmotsu manifold, we have
    \[
       \nabla^g _X \xi= \alpha \bigl(X - \eta(X)\,\xi \bigr),
    \]
    for every $X\in \mathfrak{X}(M)$. Consequently, $\alpha$ is determined by the $\alpha$-Kenmotsu structure, which is $G$-invariant, so $\alpha$ is $G$-invariant as well. Since the action is transitive, $\alpha$ is constant on $M$. Finally, the Weyl structure $\nabla$ constructed in~\ref{ex:3.3} satisfies the hypotheses of~\ref{thm:1.1}.
    Moreover, since $(M,g,\varphi,\xi,\eta)$ is homogeneous, $(M,g)$ together with $\nabla$ satisfies the hypotheses of~\ref{thm:5.5}. The result follows.
\end{proof}

\printbibliography

\enlargethispage{3\baselineskip}

\end{document}